\documentclass[12pt]{amsart}
\usepackage{exscale,aliascnt,amssymb,amsthm,amsxtra,latexsym,mathrsfs,paralist,varioref,verbatim}

\usepackage[pagebackref]{hyperref} 

\usepackage[capitalize]{cleveref} 

\usepackage[abbrev,backrefs,msc-links]{amsrefs} 

\DefineSimpleKey{bib}{article-id}
\DefineSimpleKey{bib}{category}

\newcommand{\arxiv}[1]{\href{http://arxiv.org/abs/#1}{arXiv:#1 [\BibField{category}]}}

\BibSpec{arxiv}{%
    +{}  {\PrintAuthors}                {author}
    +{,} { \textit}                     {title}
    +{}  { \PrintDatePV}                {date}
    +{,} { \eprintpages}                {pages}
    +{,} { \arxiv}                      {article-id}
    +{.} { }                            {note}
    +{.} {}                             {transition}
}

\usepackage[all]{xy} 

\theoremstyle{plain}

\newaliascnt{cor}{thm}
\newtheorem{cor}[cor]{Corollary}
\aliascntresetthe{cor}
\newaliascnt{lem}{thm}

\aliascntresetthe{lem}
\newaliascnt{prop}{thm}
\newtheorem{prop}[prop]{Proposition}
\aliascntresetthe{prop}

\theoremstyle{definition}
\newaliascnt{defn}{thm}
\newtheorem{defn}[defn]{Definition}
\aliascntresetthe{defn}

\theoremstyle{remark}

\DeclareMathOperator{\divisor}{div}
\DeclareMathOperator{\exc}{exc}

\DeclareMathOperator{\lct}{lct}
\DeclareMathOperator{\newt}{Newt}
\DeclareMathOperator{\ord}{ord}

\DeclareMathOperator{\spec}{Spec}

\newcommand*{\affspace}{\mathbb{A}}
\newcommand*{\multideal}{\mathfrak{J}}
\newcommand*{\sheaf}[1]{\mathcal#1} 
\newcommand*{\termideal}{\mathfrak{t}}

\newcommand*{\NN}{\mathbb{N}}
\newcommand*{\QQ}{\mathbb{Q}}
\newcommand*{\RR}{\mathbb{R}}
\newcommand*{\ZZ}{\mathbb{Z}}

\begin{document}

\title{A short note on the multiplier ideals of monomial space curves}
\date{\today}
\author{Howard M Thompson}
\address{Mathematics Department\\ 402 Murchie Science Building \\ 303 East Kearsley Street \\ Flint, MI 48502-1950}
\email{\href{mailto:hmthomps@umflint.edu}{hmthomps@umflint.edu}}
\urladdr{\url{http://homepages.umflint.edu/~hmthomps/}}
\thanks{The author would like to thank Karen E. Smith and Roc\'{i}o Blanco as well as too many others to name, for putting up with his near endless and largely fruitless discussion of this topic.}
\begin{abstract}
  Thompson~(2014) exhibits a formula for the multiplier ideal with multiplier $\lambda$ of a monomial curve $C$ with ideal $I$ as an intersection of a term coming from the $I$-adic valuation, the multiplier ideal of the term ideal of $I$, and terms coming from certain specified auxiliary valuations. This short note shows it suffices to consider at most two auxiliary valuations. This improvement is achieved through a more intrinsic approach, reduction to the toric case.
\end{abstract}
\maketitle

\section{Introduction}

Let $(Y, \Delta)$ be a pair, consisting of a normal variety $Y$ over an algebraically closed field of characteristic zero and a $\QQ$-divisor $\Delta$ such that $K_Y+\Delta$ is $\QQ$-Cartier. Let $\pi:X\rightarrow Y$ be a log resolution of the ideal sheaf $\sheaf{I}\subseteq\sheaf{O}_Y$ that is also a log resolution of the pair $(Y, \Delta)$. That is, $\pi$ is a proper birational morphism such that $X$ is smooth, the union of the exceptional set of $\pi$ and $\pi^{-1}(\Delta)$ is a divisor with simple normal crossing support, and $\sheaf{I}\cdot\sheaf{O}_X=\sheaf{O}_X(-F)$ is also a divisor with simple normal crossing support. In this setting, we define the multiplier ideal of $\sheaf{I}^\lambda$ on the pair $(Y,\Delta)$ to be
\[
    \multideal\left((Y,\Delta),\sheaf{I}^\lambda\right)=\pi_*\sheaf{O}_X(K_X-\lfloor\pi^*(K_Y+\Delta)+\lambda F\rfloor).
\]
This ideal sheaf on $Y$ does not depend upon the choice of log resolution.

In recent years, researchers have begun to study which divisors on a log resolution contribute jumping numbers. See Alberich-Carrami\~{n}ana, \`{A}lvarez Montaner and Dachs-Cadefau~\cite{Alberich-Carraminana2014}, Galindo and Monserrat~\cite{MR2671187}, Hyry and J{\"a}rvilehto~\cite{MR2844818}, Naie~\cite{MR2470185}, Naie~\cite{MR3035120}, Smith and Thompson~\cite{MR2389246}, and Tucker~\cite{MR2592954}. This paper refines the result of Thompson~\cite{MR3168880} by finding a smaller set of divisors that contains all the divisors that contribute jumping numbers for a monomial space curve.

\vref{S:factorizing} of this paper recalls a strengthening of the notion of an embedded resolution of singularities known as a factorizing resolution and uses it to provide a proposition (\vref{P:b-ideal}) about the structure of multiplier ideals. 

\vref{S:toric_case} of this paper recalls the Howald-Blickle Theorem (\vref{P:Blickle}) that provides a formula for the multiplier ideals of a monomial ideal on a normal affine toric variety, provides a reinterpretation (\vref{P:toroidal}) of that theorem, and provides a formula (\vref{P:principal}) for the multiplier ideals of a principal binomial ideal.

\vref{S:space_curves} applies the ideas of the previous sections to refine the result of Thompson~\cite{MR3168880}.

\section{Using factorizing resolutions to compute multiplier ideals} \label{S:factorizing}

\begin{defn}
  Let $Z$ be a generically smooth subscheme of any variety $Y$. A \emph{factorizing resolution} of $Z$ is an embedded resolution $\pi:X\rightarrow Y$ of $Z$ such that
  \[
    \sheaf{I}_Z\cdot\sheaf{O}_X=\sheaf{I}_{\widetilde{Z}}\cdot\sheaf{L}
  \]
  where $\widetilde{Z}$ is the strict transform of $Z$, $\sheaf{L}$ is an invertible sheaf, and the support of $\sheaf{I}_Z\cdot\sheaf{O}_X$ is a simple normal crossings variety. 
\end{defn}

Recall that $\pi$ is an embedded resolution of $Z$ if it is proper birational morphism $\pi:X\rightarrow Y$ such that: $X$ is smooth and $\pi$ is an isomorphism over the generic points of the components of $Z$, the exceptional locus $\exc(\pi)$ of $\pi$ is a divisor with simple normal crossing support, and the strict transform $\widetilde{Z}$ is smooth and transverse to $\exc(\pi)$. For an embedded resolution, we always have $\sheaf{I}_Z\cdot\sheaf{O}_X=\sheaf{I}_{\widetilde{Z}}\cap\sheaf{L}$ for some invertible sheaf $\sheaf{L}$. Here we require the intersection to be a product. Typically, this is achieved by blowing up embedded components of $\sheaf{I}_Z\cdot\sheaf{O}_X$. Here is a theorem on the existence of factorizing resolutions.

\begin{prop}(Theorem~1.2 of Bravo~\cite{MR3031261}, Section~3 of Eisenstein~\cite{Eisenstein2010})
  Let $Z$ be a generically smooth subscheme of any variety $Y$ over an algebraically closed field of characteristic zero such that there exists a birational morphism $\mu_1:Y'\rightarrow Y$ from a smooth variety $Y'$ that is an isomorphism over the generic points of the components of $Z$.  If $D$ is a divisor on $Y'$ with simple simple normal crossing support such that no component of the strict transform of $Z$ is contained in $D$, then there exists a factorizing resolution $\pi:X\rightarrow Y$ of $Z$ that factors through $\mu_1$, $\pi=\mu_2\circ\mu_1$, such that $\widetilde{Z}\cup\exc(\pi)\cup\mu_2^{-1}(D)$ has simple normal crossing support.
\end{prop}

Notice that if $\pi:X\rightarrow Y$ is such a factoring resolution of $Z$, then the blowup of $\widetilde{Z}$ is a log resolution of $Z$ and that the exceptional locus of this blowup consists of a collection of prime divisors in one-to-one correspondence with the components of $Z$ with codimension at least two. 

\begin{prop} \label{P:b-ideal}
  Let $Z_1,\ldots,Z_r$ be the components of $Z$ and suppose $e_i$ is the codimension of $Z_i$ for all $i$. Fix a factorizing resolution $\pi:X\rightarrow Y$ of $Z$ that is also a log resolution of the pair $(Y,\Delta)$ and let $\mathfrak{b}=\pi_*(\sheaf{L})$ where $\sheaf{I}_Z\cdot\sheaf{O}_X=\sheaf{I}_{\widetilde{Z}}\cdot\sheaf{L}$ as above. Then,
  \[
    \multideal\left((Y,\Delta),\sheaf{I}_Z^\lambda\right)=\multideal\left((Y,\Delta),\mathfrak{b}^\lambda\right)\cap\bigcap_{i=1}^r\sheaf{I}_{Z_i}^{(\lfloor\lambda+1-e_i\rfloor)}
  \]
\end{prop}

\begin{proof}
  Since $\sheaf{I}_Z\subseteq\mathfrak{b}$, it is clear that $\multideal\left((Y,\Delta),\sheaf{I}_Z^\lambda\right)\subseteq\multideal\left((Y,\Delta),\mathfrak{b}^\lambda\right)$. Let us now show $\multideal\left((Y,\Delta),\sheaf{I}_Z^\lambda\right)\subseteq\sheaf{I}_{Z_i}^{(\lfloor\lambda+1-e_i\rfloor)}$ for each $i$. Since $Z$ is generically smooth, the $\sheaf{I}_{Z_i}$ are prime and the $\sheaf{I}_{Z_i}^{(\lfloor\lambda+1-e_i\rfloor)}$ are primary. Since $\sheaf{I}_{Z_i}^{(\lfloor\lambda+1-e_i\rfloor)}$ is primary, it suffices to check the inclusion generically along the corresponding component (that is, after localizing at $\sheaf{I}_{Z_i}$). Because $Z$ is generically reduced, $\sheaf{I}_Z\cdot\sheaf{O}_{Z_i}=\sheaf{I}_{Z_i}\cdot\sheaf{O}_{Z_i}$. It is simple calculation based on the fact that one can resolve $\sheaf{O}_{Z_i}$ generically by blowing up $\sheaf{O}_{Z_i}$, and the relative canonical divisor for this blowup is $(e_i-1)E_i$, where $E_i$ is the resulting exceptional divisor.

  On the other hand, the extension of the contraction of an ideal is contained in the ideal. So, $\mathfrak{b}\cdot\sheaf{O}_X\subseteq\sheaf{L}$. Thus,
  \[
    \mathfrak{b}\cdot\sheaf{O}_X\cap\bigcap_{i=1}^r\sheaf{I}_{\widetilde{Z}_i}\subseteq\sheaf{L}\cap\bigcap_{i=1}^r\sheaf{I}_{\widetilde{Z}_i}=\sheaf{I}_Z\cdot\sheaf{O}_X.
  \]
  Therefore, we see
  \[
    \multideal\left((Y,\Delta),\mathfrak{b}^\lambda\right)\cap\bigcap_{i=1}^r\sheaf{I}_{Z_i}^{(\lfloor\lambda+1-e_i\rfloor)}\subseteq\multideal\left((Y,\Delta),\sheaf{I}_Z^\lambda\right)
  \]
  using any log resolution of $\mathfrak{b}$, $Z$ and $(Y,\Delta)$ that factors through $\pi$.
\end{proof}

\section{Exploiting the toric case} \label{S:toric_case}

This paragraph is, essentially, a direct quote of Blickle~\cite{MR2092724}. Let $(Y,\Delta)$ be a pair such that $Y$ is a normal (affine) toric variety (say $Y=\spec R$ for some normal semigroup ring $R\subseteq \Bbbk[x_1^{\pm1},\ldots,x_n^{\pm1}]$) and $\Delta$ is a torus invariant $\QQ$-divisor. Since $K_Y+\Delta$ is $\QQ$-Cartier and torus invariant, there is a monomial $\mathbf{x}^\mathbf{u}$ such that $\divisor\mathbf{x}^\mathbf{u}=r(K_Y+\Delta)$ for some integer $r$. Set $\mathbf{w}=\mathbf{u}/r$. Blickle's version of Howald's~\cite{MR1828466} formula is the following.

\begin{prop}(Theorem~1 of Blickle~\cite{MR2092724}) \label{P:Blickle}
  Let $\mathfrak{a}$ be a monomial ideal on $Y$. Then, if $\newt(\mathfrak{a})$ is the Newton Polyhedron of $\mathfrak{a}$,
  \[
    \multideal\left((Y,\Delta),\mathfrak{a}^\lambda\right)=\langle\mathbf{x}^\mathbf{v}\in R\mid\mathbf{v}+\mathbf{w}\in\text{ interior of }\lambda\newt(\mathfrak{a})\rangle
  \]
  for all $\lambda > 0$. 
\end{prop}

This means that the multiplier ideals of a monomial ideal on a toric variety are contributed by (divisors supported on unions of) the Rees divisors of the ideal. (See Thompson~\cite{Thompson2003} for  a quick overview of the relationship between toric blowups and Newton polyhedra.) Other divisors that may appear on a log resolution do not contribute.

\begin{defn}
  We will say a sheaf $\sheaf{F}$ (respectively a Weil divisor $D$) on $X$ is \emph{locally monomial} if $X$ can be covered by open subschemes $U$ such that each $U$ is isomorphic to an open subscheme of a normal toric variety in such a way that $\sheaf{F}(U)$ is identified with a torus invariant sheaf (resp. a torus invariant divisor).
\end{defn}

\begin{prop} \label{P:toroidal}
  Let $(Y, \Delta)$ be a pair, consisting of a normal variety $Y$ over an algebraically closed field of characteristic zero and a $\QQ$-divisor $\Delta$ such that $K_Y+\Delta$ is $\QQ$-Cartier. If $\pi:X\rightarrow Y$ is a proper birational morphism such that $\pi^{-1}(\Delta)$ and $\sheaf{I}\cdot\sheaf{O}_X$ are locally monomial, then the multiplier ideals of $\sheaf{I}$ are contributed by the Rees divisors of $\sheaf{I}\cdot\sheaf{O}_X$.
\end{prop}

\begin{proof}
  Since the question is local on $X$, it suffices to consider the case where $X$ is a normal affine toric variety and $\mathfrak{a}=\sheaf{I}\cdot\sheaf{O}_X$ is a monomial ideal. Let $\mu:X'\rightarrow X$ be a toric log resolution of the ideal $\mathfrak{a}$ that is also a log resolution of the pair $(X, \pi^{-1}(\Delta))$. Evidently, $\mu$ factors through the blowup of $\mathfrak{a}$ and, as in the toric case, orders of vanishing on any exceptional divisor of $\mu\circ\pi$ are determined by those on the blowup of $\mathfrak{a}$. This is just the fact that when one represents a polyhedron as an intersection of half-spaces it suffices to consider only the  facet-defining half-spaces.
\end{proof}

Consider the case of any principal binomial ideal $I=\langle\mathbf{x}^{\mathbf{v}_1}-\mathbf{x}^{\mathbf{v}_2}\rangle\subseteq\Bbbk[x_1,\dots,x_n]$. 

\begin{prop} \label{P:principal}
  If $I=\langle\mathbf{x}^{\mathbf{v}_1}-\mathbf{x}^{\mathbf{v}_2}\rangle$ and $\termideal=\langle\mathbf{x}^{\mathbf{v}_1},\mathbf{x}^{\mathbf{v}_2}\rangle$ is the term ideal of $I$, then
  \[
    \multideal(\affspace^n,I^\lambda)=\multideal(\affspace^n,\termideal^\lambda)\cap I^{(\lfloor\lambda\rfloor)}.
  \]
\end{prop}

\begin{proof}
  Let $\mathbf{v}$ be a primitive lattice vector such that $r\mathbf{v}=\mathbf{v}_1-\mathbf{v}_2$ for some positive integer $r$. Cover the normalized blowup of $\termideal$ with affine open toric varieties $U_1,\ldots,U_s$. Now, consider the covering consisting of the open sets of the form $\spec\Bbbk[x_1,\dots,x_n,\mathbf{x}^{\pm\mathbf{v}}] \setminus V(f)$ where $f=\frac{\mathbf{x}^{r\mathbf{v}}-1}{\mathbf{x}^\mathbf{v}-\zeta}$ for an $r$th root of unity $\zeta$ and the open subsets obtained by removing the closure of $V(\mathbf{x}^{r\mathbf{v}}-1)\subseteq\spec\Bbbk[x_1^{\pm1},\dots,x_n^{\pm1}]$ from each $U_i$. 

  Note that since $\mathbf{v}$ is primitive, $\ZZ\mathbf{v}$ splits from $\ZZ^n$. So, $\NN^n+\ZZ\mathbf{v}\cong S\times\ZZ$ where $S$ is the image of $\NN^n$ in $\ZZ^n/\ZZ\mathbf{v}$. Each open set of the form $\spec\Bbbk[x_1,\dots,x_n,\mathbf{x}^{\pm\mathbf{v}}] \setminus V(f)$ is of the form $\spec\Bbbk[S][t]_{t+1}$ where $t=\mathbf{x}^{\mathbf{v}}-1$. 

  And, on each open set of the form $U_i\setminus V(\mathbf{x}^{r\mathbf{v}}-1)$, $I\cdot\sheaf{O}_X=\termideal\cdot\sheaf{O}_X$ is monomial already. Thus, \vref{P:toroidal} applies. Moreover, it is clear that the components of the closure of $V(\mathbf{x}^{r\mathbf{v}}-1)\subseteq\spec\Bbbk[x_1^{\pm1},\dots,x_n^{\pm1}]$ are smooth and meet the boundary transversely. 

  So, a toric desingularization of the blowup of $\termideal\cdot\sheaf{O}_X$ is a factorizing resolution of $I$. Now, apply \vref{P:b-ideal}.
\end{proof}

This result is not new. Principal binomial ideals are nondegenerate. For an alternate proof, see Howald~\cite{Howald2003}.

\section{Application to the monomial space curve case} \label{S:space_curves}

Using the previous ideas, one can refine the result of Thompson~\cite{MR3168880}. The case where the the monomial space curve is contained in a smooth toric surface follows from the principal toric case by using adjunction and inversion of adjunction.

Let $\Bbbk$ be a field of characteristic zero, let $C=\{(t^{n_1},t^{n_2},t^{n_3})\}\subset\affspace_\Bbbk^3$ be a monomial space curve not contained in a smooth toric surface. Assume $\mathbf{n}=\begin{bmatrix}n_1 & n_2 & n_3\end{bmatrix}\in\ZZ_{>0}^3$ is a primitive vector, let $\ord_\mathbf{n}$ be the monomial valuation given by the standard pairing, $\mathbf{x}^\mathbf{m}\mapsto\langle\mathbf{n},\mathbf{m}\rangle$, and let $I\subset\Bbbk[x_1,x_2,x_3]$ be the ideal of $C$. We may assume there exist irreducible binomials $f_1=\mathbf{x}^{\mathbf{m}_1^+}-\mathbf{x}^{\mathbf{m}_1^-}$, $f_2=\mathbf{x}^{\mathbf{m}_2^+}-\mathbf{x}^{\mathbf{m}_2^-}$, and  $f_3=\mathbf{x}^{\mathbf{m}_3^+}-\mathbf{x}^{\mathbf{m}_3^-}$ such that $\{f_1,f_2,f_3\}$ or $\{f_1,f_2\}$ is a minimal generating set for $I$. Let $\termideal=\left(\mathbf{x}^{\mathbf{m}_1^+},\mathbf{x}^{\mathbf{m}_1^-},\mathbf{x}^{\mathbf{m}_2^+},\mathbf{x}^{\mathbf{m}_2^-},\mathbf{x}^{\mathbf{m}_3^+},\mathbf{x}^{\mathbf{m}_3^-}\right)$ be the term ideal of $I$. Let $d_i=\ord_\mathbf{n}(f_i)$ for $i=1,2,3$. Order the generators so that $d_1< d_2< d_3$ and order the $n_i$ so that $n_i|d_i$ for $i=1,2$ (and $n_3|d_3$ when $f_3$ is a minimal generator). See Section~3 of Shibuta and Takagi~\cite{MR2533766} for a more detailed setup. Let $\mathbf{m}_1=\mathbf{m}_1^+-\mathbf{m}_1^-$ and let $\mathbf{q}=\begin{bmatrix}q_1 & q_2 & 0\end{bmatrix}\in\NN^3$ be the primitive vector such that $\langle\mathbf{q},\mathbf{m}_1\rangle=0$. And, let $e_i=\ord_\mathbf{q}(f_i)$ for $i=1,2,3$. 

\begin{prop} \label{P:curvefix}
  Let $\mathfrak{a}_1=\left(\mathbf{x}^{\mathbf{m}_1^+},\mathbf{x}^{\mathbf{m}_1^-}\right)$, let $\mathfrak{a}_2=(x_1^{n_2n_3},x_2^{n_1n_3},x_3^{n_1n_2})$, and let the toric variety $X=X_\Sigma$ be the normalized blowup of $\mathfrak{a}_1\mathfrak{a}_2$. Then the ideal sheaf $I\cdot\sheaf{O}_X$ is locally monomial.
\end{prop}

\begin{proof}
  The blowup of $\mathfrak{a}_1$ is the partial desingularization of the toric surface $V(f_1)$ identified in Gonz{\'a}lez P{\'e}rez and Teissier~\cite{MR1892938}, and the normalized blowup of $\mathfrak{a}_2$ is  the partial desingularization of $C$. The fan $\Sigma_1$ of the blowup of $\mathfrak{a}_1$ has two maximal cones $\left\{\mathbf{v}\in\RR_{\geq0}^3\mid\langle\mathbf{v},\mathbf{m}_1\rangle\leq0\right\}$ and $\left\{\mathbf{v}\in\RR_{\geq0}^3\mid\langle\mathbf{v},\mathbf{m}_1\rangle\geq0\right\}$. The normalized blowup of $\mathfrak{a}_2$ is stellar subdivision along the ray $\rho=\RR_{\geq0}\mathbf{n}$. Note that $\mathbf{n}$ is in the intersection of the two maximal cones of $\Sigma_1$. So, the two operations on fans, stellar subdivision along $\mathbf{n}$ and cutting with the plane $\left\{v\in\RR_{\geq0}^3\mid\langle\mathbf{v},\mathbf{m}_1\rangle=0\right\}$ commute. And, $\Sigma$ is the stellar subdivision along $\rho$ of $\Sigma_1$. (Any toric desingularization of $X$ provides a common embedded desingularization of $C$ and the surface $V(f_1)$.)

  First, consider the affine open $U_\rho$ of $X$ and fix an element of the affine semigroup $\mathbf{m}_\rho\in\mathsf{S}_\rho$ such that $\langle\mathbf{m}_\rho,\mathbf{n}\rangle=1$. I claim, $\{\mathbf{m}_1,\mathbf{m}_2\}$ is a basis of the kernel of the matrix $\begin{bmatrix}n_1 & n_2 & n_3\end{bmatrix}$, $\mathsf{S}_\sigma=\NN^3+\ZZ\mathbf{m}_1+\ZZ\mathbf{m}_2$, $f_i=(\mathbf{x}^{\mathbf{m}_\rho})^{d_i}(\mathbf{x}^{\mathbf{m}_i}-1)$ for each $i=1,2,3$, and $f_3\in(f_1,f_2)\Bbbk[\mathsf{S}_\sigma]$. So,
  \begin{multline*}
    I\cdot\sheaf{O}_{U_\rho}=\left((\mathbf{x}^{\mathbf{m}_\rho})^{d_1}(\mathbf{x}^{\mathbf{m}_1}-1),(\mathbf{x}^{\mathbf{m}_\rho})^{d_2}(\mathbf{x}^{\mathbf{m}_2}-1)\right) \\
     =\left(\mathbf{x}^{d_1\mathbf{m}_\rho}\right)\cap\left(\mathbf{x}^{\mathbf{m}_1}-1,\mathbf{x}^{d_2\mathbf{m}_\rho}\right)\cap(\mathbf{x}^{\mathbf{m}_1}-1,\mathbf{x}^{\mathbf{m}_2}-1)
  \end{multline*}
  is monomial in $\mathbf{x}^{\mathbf{m}_\rho}$, $\mathbf{x}^{\mathbf{m}_1}-1$, and $\mathbf{x}^{\mathbf{m}_2}-1$. Since $d_1<d_2$, there is an embedded component supported on the intersection of the strict transform of the surface $\widetilde{V(f_1)}$ and the divisor $D_\rho$. Away from this embedded component, $I\cdot\sheaf{O}_X=\termideal\cdot\sheaf{O}_X$. Thus, it suffices to check the closed points where the curve $\overline{V(\mathbf{x}^{\mathbf{m}_1}-1,\mathbf{x}^{\mathbf{m}_\rho})}$ meets $X\setminus U_\rho$. 

  Let $p\in X$ be one of these two points, and let $\sigma$ be the smallest cone of $\Sigma$ such that $p\in U_\sigma$. Evidently, $\rho\subsetneq\sigma$ since $p\notin U_\rho$. After possibly replacing $\mathbf{m}_1$ with $-\mathbf{m}_1$, we may assume $\mathbf{x}^{\mathbf{m}_1}-1\in\mathfrak{m}_{X,p}$. Since $\mathbf{x}^{\mathbf{m}_1}-1\in\mathfrak{m}_{X,p}$, $p$ is not a torus-fixed point and $\sigma$ is two-dimensional. Let $\sigma=\RR_{\geq0}^2\begin{bmatrix}n_1 & n_2 & n_3 \\ r_1 & r_2 & r_3\end{bmatrix}$. Note that $\mathbf{m}_1$ is a basis for the kernel of $\begin{bmatrix}n_1 & n_2 & n_3 \\ r_1 & r_2 & r_3\end{bmatrix}$, and $\langle\mathbf{n},\mathbf{m}_i\rangle=0$ for $i=1,2,3$. After possibly replacing $\mathbf{m}_2$ with $-\mathbf{m}_2$ and $\mathbf{m}_3$ with $-\mathbf{m}_3$, we may assume $\mathbf{m}_2,\mathbf{m}_3\in\mathsf{S}_\sigma=\NN^3+\ZZ \mathbf{m}_1$. As in \vref{P:principal}, the affine semigroup is a product $\mathsf{S}_\sigma\cong\mathsf{S}\times\ZZ\mathbf{m}_1$ where $\mathsf{S}$ is the image of $\NN^3$ in the quotient $\ZZ^3/\ZZ\mathbf{m}_1$. Thus, $\Bbbk[\mathsf{S}_\sigma]\cong\Bbbk[\mathsf{S}][t_3]_{t_3+1}$ where $t_3=\mathbf{x}^{\mathbf{m}_1}-1$.
\end{proof}

Recall $\mathbf{q}=\begin{bmatrix}q_1 & q_2 & 0\end{bmatrix}\in\NN^3$ is the primitive vector such that $\langle\mathbf{q},\mathbf{m}_1\rangle=0$ and $e_i=\ord_\mathbf{q}(f_i)$ for $i=1,2,3$.  Here is the improvement to the main theorem of Thompson~\cite{MR3168880}.

\begin{prop} \label{P:main}
  \begin{enumerate}[(i)]
    \item If $I$ is a complete intersection or if $e_2(d_3-d_1)\leq e_1(d_3-d_2)$, then 
      \[
        \multideal(I^\lambda)=I^{\left(\lfloor\lambda-1\rfloor\right)}\cap\multideal(\termideal^\lambda)\cap\left(f\mid\nu_1(f)\geq\lfloor a_1\lambda-k_1\rfloor\right)
      \]
      where $\nu_1$ is the valuation given by the generating sequence $x_i\mapsto n_i$ for $i=1,2,3$, $f_1\mapsto d_2$. Thus, $a_1=\nu_1(I)=d_2$ and $k_1=\nu_1(J_{R_{\nu_1}/\Bbbk[x]})=n_1+n_2+n_3+d_2-d_1$ where $J_{R_{\nu_1}/\Bbbk[x]}$ is the Jacobian of the discrete valuation ring $R_{\nu_1}$ of $\nu_1$ over $\Bbbk[x]$.
    \item Otherwise, 
      \[
        \multideal(I^\lambda)=I^{\left(\lfloor\lambda-1\rfloor\right)}\cap\multideal(\termideal^\lambda)\bigcap_{i=1,2}\left(f\mid\nu_i(f)\geq\lfloor a_i\lambda-k_i\rfloor\right)
      \]
      where $\nu_1$ is as before and $\nu_2$ is given by the generating sequence $x_1\mapsto e_2n_1+(d_3-d_2)q_1$, $x_2\mapsto e_2n_2+(d_3-d_2)q_2$, $x_3\mapsto e_2n_3$, $f_1\mapsto e_2d_3$. Thus, $a_2=\nu_2(I)=e_2d_3$, and $k_2=\nu_2(J_{R_{\nu_2}/\Bbbk[x]})=e_2(n_1+n_2+n_3)+(d_3-d_1)(q_1+q_2)+e_2(d_3-d_1)-e_1(d_3-d_2)$ where $J_{R_{\nu_2}/\Bbbk[x]}$ is the Jacobian of the discrete valuation ring $R_{\nu_2}$ of $\nu_2$ over $\Bbbk[x]$.
  \end{enumerate}
\end{prop}

\begin{proof}
  Apply \vref{P:curvefix} and the convex geometry computation in the appendix.
\end{proof}

\begin{cor}
  \begin{enumerate}[(i)]
    \item If $I$ is a complete intersection or if $e_2(d_3-d_1)\leq e_1(d_3-d_2)$, then the log canonical threshold of $I$ (at the origin) is
      \[
        \lct_0(I)=\min\left(\lct_0(\termideal),\frac{k_1+1}{a_1}\right).
      \]
    \item Otherwise, 
      \[
        \lct_0(I)=\min\left(\lct_0(\termideal),\frac{k_1+1}{a_1},\frac{k_2+1}{a_2}\right).
      \]
  \end{enumerate}
\end{cor}

Note that when $e_2(d_3-d_1)=e_1(d_3-d_2)$, $\nu_2$ is monomial in the $\mathbf{x}$-variables and both formulas apply. In Example~5.3 of Blanco and Encinas~\cite{Blanco2014}, $e_2(d_3-d_1)=e_1(d_3-d_2)$, $\nu_2$. I do not know an example where $e_2(d_3-d_1)>e_1(d_3-d_2)$. A Macaulay2 package that implements this calculation, as presented in Thompson~\cite{MR3168880}, is described in Teitler~\cite{Teitler2013}.

\section*{Appendix}

Recall that $d_i=\ord_\mathbf{n}(f_i)$ and let $\mathbf{u}_i=\begin{bmatrix}d_i \\ e_i\end{bmatrix}=\begin{bmatrix}n_1 & n_2 & n_3 \\ r_1 & r_2 & r_3\end{bmatrix}\mathbf{m}_i^-$ for $i=1,2,3$. In the local monomial coordinates, we find the Rees valuations from the facets of the Newton polyhedron. It suffices to consider the ideal $\left(\mathbf{t}^{u_1}t_3,\mathbf{t}^{u_2},\mathbf{t}^{u_3}\right)\subset\Bbbk[\mathsf{S}]ty_3]$ (see the conclusion of the proof of \vref{P:main}). We know $e_1>e_2$ by examining Section~3 of Shibuta and Takagi~\cite{MR2533766}. And, $r_2=0$ or $r_3=0$. 

If  $r_2=0$, then $e_2=0$, $\mathbf{t}^{\mathbf{u}_3}\in(\mathbf{t}^{\mathbf{u}_2})$, and the facets of the Newton polyhedron $\newt\left(\mathbf{t}^{\mathbf{u}_1}t_3,\mathbf{t}^{\mathbf{u}_2}\right)$ are orthogonal to the rows of the matrix
\[
  \begin{bmatrix}
    1 & 0 & 0 \\
    0 & 1 & 0 \\
    0 & 0 & 1 \\
    e_1 & d_2-d_1 & 0 \\
    1 & 0 & d_2-d_1
  \end{bmatrix}
\]
This includes the complete intersection case. Note that the only two rows of our matrix that have a nonzero last entry are $\begin{bmatrix}0 & 0 & 1\end{bmatrix}$ and $\begin{bmatrix}1 & 0 & d_2-d_1\end{bmatrix}$. The other vectors correspond to valuations that are monomial in the original $\mathbf{x}$-variables. Our ideal $\left(\mathbf{t}^{\mathbf{u}_1}t_3,\mathbf{t}^{\mathbf{u}_2},\mathbf{t}^{\mathbf{u}_3}\right)$ has order zero on the valuation corresponding to $\begin{bmatrix}0 & 0 & 1\end{bmatrix}$. And, the row $\begin{bmatrix}1 & 0 & d_2-d_1\end{bmatrix}$ corresponds to $\nu_1$.

If $\mathbf{t}^{\mathbf{u}_2}\notin\overline{(\mathbf{t}^{\mathbf{u}_1},\mathbf{t}^{\mathbf{u}_3})}$ and $e_2\neq0$, then the facets of the Newton polyhedron $\newt\left(\mathbf{t}^{\mathbf{u}_1}t_3,\mathbf{t}^{\mathbf{u}_2},\mathbf{t}^{\mathbf{u}_3}\right)$ are orthogonal to the rows of the matrix
\[
  \begin{bmatrix}
    1 & 0 & 0 \\
    0 & 1 & 0 \\
    0 & 0 & 1 \\
    e_1-e_2 & d_2-d_1 & 0 \\
    1 & 0 & d_2-d_1 \\ 
    e_2 & d_3-d_2 & 0
  \end{bmatrix}
\]
and these rows all have nonnegative integer entries. In terms of the parameters introduced in Section~3 of Shibuta and Takagi~\cite{MR2533766}, $\alpha\leq\gamma$ in this case.

If $\mathbf{t}^{\mathbf{u}_2}\in\overline{(\mathbf{t}^{\mathbf{u}_1},\mathbf{t}^{\mathbf{u}_3})}$ and $e_2\neq0$, then $r_3=0$ and the facets of the Newton polyhedron $\newt\left(\mathbf{t}^{\mathbf{u}_1}t_3,\mathbf{t}^{\mathbf{u}_2},\mathbf{t}^{\mathbf{u}_3}\right)$ are orthogonal to the rows of the matrix
\[
  \begin{bmatrix}
    1 & 0 & 0 \\
    0 & 1 & 0 \\
    0 & 0 & 1 \\
    e_1-e_3 & d_3-d_1 & 0 \\
    1 & 0 & d_2-d_1 \\
    e_2 & d_3-d_2 & e_2(d_3-d_1)-e_1(d_3-d_2)
  \end{bmatrix}
\]
and these rows all have nonnegative integer entries. Note that the only three rows that have a nonzero last entry are $\begin{bmatrix}0 & 0 & 1\end{bmatrix}$, $\begin{bmatrix}1 & 0 & d_2-d_1\end{bmatrix}$, and
\[
  \begin{bmatrix}e_2 & d_3-d_2 & e_2(d_3-d_1)-e_1(d_3-d_2)\end{bmatrix}
\]
corresponding to the only bounded facet of  $\newt\left(\mathbf{t}^{\mathbf{u}_1}t_3,\mathbf{t}^{\mathbf{u}_2},\mathbf{t}^{\mathbf{u}_3}\right)$. The other vectors correspond to valuations that are monomial in the original $\mathbf{x}$-variables. And, the bounded facet corresponds to $\nu_2$. For $\nu_2$, the orders of vanishing of the $x$-variables are given by the entries of
\begin{multline*}
  \begin{bmatrix}
    e_2 & d_3-d_2 & e_2(d_3-d_1)-e_1(d_3-d_2)
  \end{bmatrix}
  \begin{bmatrix}
    n_1 & n_2 & n_3 \\ 
    q_1 & q_2 & 0   \\
    0   & 0   & 0
  \end{bmatrix} \\
  =
  \begin{bmatrix}
    e_2n_1+(d_3-d_2)q_1 & e_2n_2+(d_3-d_2)q_2 & e_2n_3
  \end{bmatrix}
\end{multline*}
and $\nu_2(f_i)=e_2d_3$ for all $i=1,2,3$
\begin{multline*}
  \begin{bmatrix}
    e_2 & d_3-d_2 & e_2(d_3-d_1)-e_1(d_3-d_2)
  \end{bmatrix}
  \begin{bmatrix}
    d_1 & d_2 & d_3 \\ 
    e_1 & e_2 & 0   \\
    1   & 0   & 0
  \end{bmatrix} \\
  =
  \begin{bmatrix}
    e_2d_3 & e_2d_3 & e_2d_3
  \end{bmatrix}
\end{multline*}


\begin{bibdiv}
\begin{biblist}

\bib{Alberich-Carraminana2014}{arxiv}{
   author={Alberich-Carrami\~{n}ana, Maria},
   author={\`{A}lvarez Montaner, Josep},
   author={Dachs-Cadefau, Ferran},
   title={Multiplier ideals in two-dimensional local rings with rational singularities},
   date={2014},
   pages={32},
   eprint={http://arxiv.org/abs/1412.3605},
   article-id={1412.3605v1},
   category={math.AG},
}

\bib{Blanco2014}{arxiv}{
   author={Blanco, Roc{\'\i}o},
   author={Encinas, Santiago},
   title={A procedure for computing the log canonical threshold of a binomial ideal},
   date={2014},
   pages={28},
   eprint={http://arxiv.org/abs/1405.3942},
   article-id={1405.3942v2},
   category={math.AG},
}

\bib{MR3031261}{article}{
   author={Bravo, Ana},
   title={A remark on strong factorizing resolutions},
   journal={Rev. R. Acad. Cienc. Exactas F\'\i s. Nat. Ser. A Math. RACSAM},
   volume={107},
   date={2013},
   number={1},
   pages={53--60},
   issn={1578-7303},
   review={\MR{3031261}},
   doi={10.1007/s13398-012-0080-8},
}

\bib{MR2092724}{article}{
   author={Blickle, Manuel},
   title={Multiplier ideals and modules on toric varieties},
   journal={Math. Z.},
   volume={248},
   date={2004},
   number={1},
   pages={113--121},
   issn={0025-5874},
   review={\MR{2092724 (2006a:14082)}},
   doi={10.1007/s00209-004-0655-y},
}

\bib{Eisenstein2010}{arxiv}{
   author={Eisenstein, Eugene},
   title={Generalizations of the restriction theorem for multiplier ideals},
   date={2010},
   pages={17},
   eprint={http://arxiv.org/abs/1001.2841},
   article-id={1001.2841v1},
   category={math.AG},
}

\bib{MR2671187}{article}{
   author={Galindo, Carlos},
   author={Monserrat, Francisco},
   title={The Poincar\'e series of multiplier ideals of a simple complete
   ideal in a local ring of a smooth surface},
   journal={Adv. Math.},
   volume={225},
   date={2010},
   number={2},
   pages={1046--1068},
   issn={0001-8708},
   review={\MR{2671187 (2012a:14039)}},
   doi={10.1016/j.aim.2010.03.008},
}

\bib{MR1892938}{article}{
   author={Gonz{\'a}lez P{\'e}rez, Pedro Daniel},
   author={Teissier, Bernard},
   title={Embedded resolutions of non necessarily normal affine toric
   varieties},
   language={English, with English and French summaries},
   journal={C. R. Math. Acad. Sci. Paris},
   volume={334},
   date={2002},
   number={5},
   pages={379--382},
   issn={1631-073X},
   review={\MR{1892938 (2003b:14019)}},
   doi={10.1016/S1631-073X(02)02273-2},
}

\bib{MR1828466}{article}{
   author={Howald, J. A.},
   title={Multiplier ideals of monomial ideals},
   journal={Trans. Amer. Math. Soc.},
   volume={353},
   date={2001},
   number={7},
   pages={2665--2671 (electronic)},
   issn={0002-9947},
   review={\MR{1828466 (2002b:14061)}},
   doi={10.1090/S0002-9947-01-02720-9},
}

\bib{Howald2003}{arxiv}{
   author={Howald, J. A.},
   title={Multiplier Ideals of Sufficiently General Polynomials},
   date={2003},
   pages={9},
   eprint={http://arxiv.org/abs/math/0303203},
   article-id={math/0303203v1},
   category={math.AG},
}

\bib{MR2844818}{article}{
   author={Hyry, Eero},
   author={J{\"a}rvilehto, Tarmo},
   title={Jumping numbers and ordered tree structures on the dual graph},
   journal={Manuscripta Math.},
   volume={136},
   date={2011},
   number={3-4},
   pages={411--437},
   issn={0025-2611},
   review={\MR{2844818}},
   doi={10.1007/s00229-011-0449-6},
}

\bib{MR2470185}{article}{
   author={Naie, Daniel},
   title={Jumping numbers of a unibranch curve on a smooth surface},
   journal={Manuscripta Math.},
   volume={128},
   date={2009},
   number={1},
   pages={33--49},
   issn={0025-2611},
   review={\MR{2470185 (2009j:14034)}},
   doi={10.1007/s00229-008-0223-6},
}

\bib{MR3035120}{article}{
   author={Naie, Daniel},
   title={Mixed multiplier ideals and the irregularity of abelian coverings
   of smooth projective surfaces},
   journal={Expo. Math.},
   volume={31},
   date={2013},
   number={1},
   pages={40--72},
   issn={0723-0869},
   review={\MR{3035120}},
   doi={10.1016/j.exmath.2012.08.005},
}

\bib{MR2533766}{article}{
   author={Shibuta, Takafumi},
   author={Takagi, Shunsuke},
   title={Log canonical thresholds of binomial ideals},
   journal={Manuscripta Math.},
   volume={130},
   date={2009},
   number={1},
   pages={45--61},
   issn={0025-2611},
   review={\MR{2533766}},
   doi={10.1007/s00229-009-0270-7},
}
    
\bib{MR2389246}{article}{
   author={Smith, Karen E.},
   author={Thompson, Howard M.},
   title={Irrelevant exceptional divisors for curves on a smooth surface},
   conference={
      title={Algebra, geometry and their interactions},
   },
   book={
      series={Contemp. Math.},
      volume={448},
      publisher={Amer. Math. Soc., Providence, RI},
   },
   date={2007},
   pages={245--254},
   review={\MR{2389246 (2009c:14004)}},
   doi={10.1090/conm/448/08669},
}

\bib{Teitler2013}{arxiv}{
   author={Teitler, Zach},
   title={Software for multiplier ideals},
   date={2013},
   pages={7},
   eprint={http://arxiv.org/abs/1305.4435},
   article-id={1305.4435v1},
   category={math.AG},
}

\bib{Thompson2003}{arxiv}{
   author={Thompson, Howard M.},
   title={Comments on toric varieties},
   date={2003},
   pages={6},
   eprint={http://arxiv.org/abs/math/0310336},
   article-id={0310336},
   category={math.AG},
}

\bib{MR3168880}{article}{
   author={Thompson, Howard M.},
   title={Multiplier ideals of monomial space curves},
   journal={Proc. Amer. Math. Soc. Ser. B},
   volume={1},
   date={2014},
   pages={33--41},
   issn={2330-1511},
   review={\MR{3168880}},
   doi={10.1090/S2330-1511-2014-00001-8},
}

\bib{MR2592954}{article}{
   author={Tucker, Kevin},
   title={Jumping numbers on algebraic surfaces with rational singularities},
   journal={Trans. Amer. Math. Soc.},
   volume={362},
   date={2010},
   number={6},
   pages={3223--3241},
   issn={0002-9947},
   review={\MR{2592954 (2011c:14106)}},
   doi={10.1090/S0002-9947-09-04956-3},
}

\end{biblist}
\end{bibdiv}
\end{document}